\documentclass[smallextended]{svjour3}%
\usepackage{graphicx}
\usepackage{amssymb}
\usepackage{amsmath}
\usepackage{amsfonts}
\usepackage{color}%
\providecommand{\U}[1]{\protect\rule{.1in}{.1in}}
\smartqed
\newtheorem{condition}{Condition}
\newtheorem{algorithm}[definition]{Algorithm}
\numberwithin{lemma}{section}
\numberwithin{proposition}{section}
\numberwithin{theorem}{section}
\numberwithin{definition}{section}
\allowdisplaybreaks
\begin{document}

\title{Bounded Perturbation Resilience of Projected Scaled Gradient Methods }
\author{Wenma Jin
\and Yair Censor
\and Ming Jiang}
\authorrunning{W. Jin, Y. Censor and M. Jiang}
\institute{
Wenma Jin
\at
LMAM, School of Mathematical Sciences, Peking University, Beijing 100871, China.
\and
Yair Censor \at
Department of Mathematics, University of Haifa, Mt. Carmel, Haifa 3498838, Israel.
\and
Ming Jiang$^{1,2}$ \at
$^{1}$LMAM, School of Mathematical Sciences, and Beijing International Center
for Mathematical Research, Peking University, Beijing 100871, China. \\
$^{2}$Cooperative Medianet Innovation Center, Shanghai Jiao Tong University,
Shanghai 200240, China.
}
\date{Received: February 2, 2014 /Revised: April 18, 2015 / Accepted: date}
\maketitle

\begin{abstract}
We investigate projected scaled gradient (PSG) methods for convex minimization
problems. These methods perform a descent step along a diagonally scaled
gradient direction followed by a feasibility regaining step via orthogonal
projection onto the constraint set. This constitutes a generalized algorithmic
structure that encompasses as special cases the gradient projection method,
the projected Newton method, the projected Landweber-type methods and the
generalized Expectation-Maximization (EM)-type methods. We prove the
convergence of the PSG methods in the presence of bounded perturbations. This
resilience to bounded perturbations is relevant to the ability to apply the
recently developed superiorization methodology to PSG methods, in particular
to the EM algorithm.


\end{abstract}

\titlerunning{Bounded Perturbation Resilience of Projected Scaled Gradient Methods}

\authorrunning{W. Jin, Y. Censor and M. Jiang}

\institute{
Wenma Jin \at
LMAM, School of Mathematical Sciences, Peking University, Beijing 100871, China.
\and
Yair Censor \at
Department of Mathematics, University of Haifa, Mt.\ Carmel, Haifa 3498838, Israel.
\and
Ming Jiang \at
LMAM, School of Mathematical Sciences, \& Beijing International Center for Mathematical Research,
Peking University, Beijing 100871, China.
}


\section{Introduction\label{sect:introduction}}

In this paper we consider convex minimization problems of the form
\begin{equation}
\left\{
\begin{array}
[c]{ll}%
\operatorname*{minimize} & \displaystyle J(x)\\
\text{subject to} & \displaystyle x\in\Omega.
\end{array}
\right.  \label{eq:convex-minimization}%
\end{equation}
The constraint set $\Omega\subseteq\mathbb{R}^{n}$ is assumed to be nonempty,
closed and convex, and the objective function $J:\Omega\mapsto\mathbb{R}$ is
convex. Many problems in engineering and technology can be modeled by
(\ref{eq:convex-minimization}). Gradient-type iterative methods are advocated
techniques for such problems and there exists an extensive literature
regarding projected gradient or subgradient methods as well as their
incremental variants, see, e.g.,
\cite{Bertsekas1999Nonlinear,HelouNeto2009Incremental,Kiwiel2004Convergence,Nesterov2004Introductory,Polyak1987Introduction}%
.

In particular, the weighted Least-Squares (LS) and the Kullback-Leibler (KL)
distance (also known as \textit{I}-divergence or cross-entropy
\cite{Csiszar1991Why}), which are two special instances of the Bregman
distances \cite[p. 33]{Censor1997Parallel}, are generally adopted as proximity
functions measuring the constraints-compatibility in the field of image
reconstruction from projections
\cite{Byrne2001Proximity,Byrne1993Iterative,Combettes1994Inconsistent,Jiang2001Development}%
. Minimization of the LS or the KL distance with additional constraints, such
as nonnegativity, naturally falls within the scope of
(\ref{eq:convex-minimization}). Correspondingly, the Landweber iteration
\cite{Landweber1951iteration} is a general gradient method for weighted LS
problems \cite[Section 6.2]{Bertero1998Introduction}, \cite[Section
4.6]{Cegielski2013Iterative}, \cite{Jiang2003Convergence},
\cite{McCormick1975uniform}, \cite{Piana1997Projected}, while the class of
expectation-maximization (EM) algorithms \cite{Shepp1982Maximum} are
essentially scaled gradient methods for the minimization of KL distance
\cite{Bertero2008Iterative,HelouNeto2005Convergence,Lanteri2001general}.

Motivated by the scaled gradient formulation of EM-type algorithms, we focus
our attention on the family of projected scaled gradient (PSG) methods, the
basic iterative step of which is given by
\begin{equation}
\label{eq:PSG-introduction}x^{k+1} := P_{\Omega}(x^{k} -\tau_{k} D(x^{k})
\nabla J(x^{k})),
\end{equation}
where $\tau_{k}$ denotes the stepsize, $D(x^{k})$ is a diagonal scaling matrix
and $P_{\Omega}$ is the orthogonal (Euclidean least distance) projection onto
$\Omega$. To our knowledge, the PSG methods presented here date back to
\cite[Eq. (29)]{Bertsekas1976Goldstein} and they resemble the projected Newton
method studied in \cite{Bertsekas1982Projected}.

From the algorithmic structural point of view, the family of PSG methods
includes, but is not limited to, the Goldstein-Levitin-Polyak gradient
projection method
\cite{Bertsekas1976Goldstein,Goldstein1964Convex,Levitin1966Constrained}, the
projected Newton method \cite{Bertsekas1982Projected}, and the projected
Landweber method \cite[Section 6.2]{Bertero1998Introduction},
\cite{Piana1997Projected}, as well as generalized EM-type methods
\cite{HelouNeto2005Convergence,Lanteri2001general}. The PSG methods should be
distinguished from the scaled gradient projection (SGP) methods in the
literature \cite{Bertero2008Iterative,Bonettini2009scaled}. PSG methods belong
to the class of two-metric projection methods \cite{Gafni1984Two}, which adopt
different norms for the computation of the descent direction and the
projection operation while SGP methods utilize the same norm for both.

The main purpose of this paper is to investigate the convergence behavior of
PSG methods and their bounded perturbation resilience. This is inspired by the
recently developed superiorization methodology (SM)
\cite{Censor2010Perturbation,Censor2013Projected,Herman2012Superiorization}.
The superiorization methodology works by taking an iterative algorithm,
investigating its perturbation resilience, and then, using proactively such
permitted perturbations, forcing\ the perturbed algorithm to do something
useful in addition to what it is originally designed to do. The original
unperturbed algorithm is called the \textquotedblleft Basic
Algorithm\textquotedblright\ and the perturbed algorithm is called the
\textquotedblleft Superiorized Version\ of the Basic
Algorithm\textquotedblright.

If the original algorithm\footnote{We use the term \textquotedblleft
algorithm\textquotedblright\ for the iterative processes discussed here, even
for those that do not include any termination criterion. This does not create
any ambiguity because whether we consider an infinite iterative process or an
algorithm with a termination rule is always clear from the context.} is
computationally efficient and useful in terms of the application at hand, and
if the perturbations are simple and not expensive to calculate, then the
advantage of this methodology is that, for essentially the same computational
cost of the original Basic Algorithm, we are able to get something more by
steering its iterates according to the perturbations.

This is a very general principle, which has been successfully used in some
important practical applications and awaits to be implemented and tested in
additional fields; see, e.g., the recent papers \cite{rand-conmath,sh14}, for
applications in intensity-modulated radiation therapy and in nondestructive
testing. The principles of superiorization and perturbation resilience along
with many references to works in which they were used, are reviewed in the
recent \cite{censor-weak-14} and \cite{gth-sup4IA}. A chronologically ordered
bibliography of scientific publications on the superiorization methodology and
perturbation resilience of algorithms has recently been compiled and is being
continuously updated by the second author. It is now available at: http://math.haifa.ac.il/yair/bib-superiorization-censor.html.

In a nutshell, the SM lies between feasibility-seeking and constrained
minimization. It is not quite trying to solve the full-fledged constrained
minimization; rather, the task is to seek a superior feasible solution in
terms of the given objective function. This can be beneficial for cases when
an exact approach to constrained minimization has not yet been discovered, or
when exact approaches are computer resources demanding or computation time
consuming. In such cases, existing feasibility-seeking algorithms that are
perturbation resilient can be turned into efficient algorithms that perform superiorization.

The basic idea of the SM originates from the discovery that some
feasibility-seeking projection algorithms for convex feasibility problems are
bounded perturbations resilient \cite{Butnariu2007Stable}. SM thus takes
advantage of the perturbation resilience property of the String-Averaging
Projections (SAP) \cite{Censor2012Convergence} or Block-Iterative Projections
(BIP) \cite{Davidi2009Perturbation,Nikazad2012Accelerated} methods to steer
the iterates of the original feasibility-seeking projection method towards a
reduced, but not necessarily minimal, value of the given objective function of
the constrained minimization problem at hand, see, e.g.,
\cite{Censor2010Perturbation,Penfold2010Total}.

The mathematical principles of the SM over general consistent
\textquotedblleft problem structures\textquotedblright\ with the notion of
bounded perturbation resilience were formulated in
\cite{Censor2010Perturbation}. The framework of the SM was extended to the
inconsistent case by using the notion of strong perturbation resilience
\cite{Herman2012Superiorization}. Most recently, the effectiveness of the SM
was demonstrated by a performance comparison with the projected subgradient
method for constrained minimization problems \cite{Censor2013Projected}.

But the SM is not limited to handling just feasibility-seeking algorithms. It
can take any \textquotedblleft Basic Algorithm\textquotedblright\ that is
bounded perturbations resilient and introduce certain permitted perturbations
into its iterates, such that the resulting algorithm is automatically steered
to produce an output that is superior with respect to the given objective
function. See Subsection \ref{subsect:BPR} below for more details on this point.

Specifically, efforts have been recently made to derive a superiorized version
of the EM algorithm, and this is why we study the bounded perturbation
resilience of the PSG methods here. Superiorization of the EM algorithm was
first reported experimentally in our previous work with application to
bioluminescence tomography \cite{Jin2013heuristic}. Such superiorized version
of the EM iteration was later applied to single photon emission computed
tomography \cite{Luo2013Superiorization}. The effectiveness of superiorization
of the EM algorithm was further validated with a study using statistical
hypothesis testing in the context of positron emission tomography
\cite{Garduno2013Superiorization}.

These efforts with regard to the EM algorithm prompted our research reported
here. Namely, the need to secure bounded perturbations resilience of the EM
algorithm that will justify the use of a superiorized version of it to seek
total variation (TV) reduced values of the image vector $x$\ in an image
reconstruction problem that employs an EM algorithm, see Section
\ref{sect:BPR-PSG} below.

The fact that the algebraic reconstruction technique (ART), see, e.g.,
\cite[Chapter 11]{Herman2009Fundamentals} and references therein, is related
to the Landweber iteration \cite{Jiang2003Convergence,Trussell1985Landweber}
for weighted LS problems and the fact that EM is essentially a scaled gradient
method for KL minimization
\cite{Bertero2008Iterative,HelouNeto2005Convergence,Lanteri2001general} prompt
us to investigate the PSG methods, which encompass both, with bounded perturbations.

So, in view of the above considerations, we ask if the convergence of PSG
methods will be preserved in the presence of bounded perturbations? In this
study, we provide an affirmative answer to this question. First we prove the
convergence of the iterates generated by
\begin{equation}
\label{eq:PSG-with-summable-perturbations-introduction}x^{k+1} := P_{\Omega
}(x^{k} -\tau_{k} D(x^{k}) \nabla J(x^{k})+ e(x^{k})),
\end{equation}
with $\{e(x^{k})\}_{k=0}^{\infty}$ denoting the sequence of outer
perturbations and satisfying
\begin{equation}
\label{eq:summable-perturbations-introduction}\sum_{k=0}^{\infty}\Vert
e(x^{k})\Vert<+\infty.
\end{equation}
This convergence result is then translated to the desired bounded perturbation
resilience of PSG methods (in Section \ref{sect:BPR-PSG} below).

The algorithmic structure of
(\ref{eq:PSG-with-summable-perturbations-introduction}%
)--(\ref{eq:summable-perturbations-introduction}) is adapted from the general
framework of the feasible descent methods studied in \cite{Luo1993Error}.
Compared with \cite{Luo1993Error}, our algorithmic extension has two aspects.
Firstly, the diagonally scaled gradient is incorporated, which allows to
include additional cases such as generalized EM-type methods. Secondly, the
perturbations in \cite{Luo1993Error} were given as
\begin{equation}
\label{eq:perturbation-by-luo}\Vert e(x^{k}) \Vert\leq\gamma\Vert x^{k} -
x^{k+1} \Vert\text{~~for~some~~} \gamma>0, \quad\forall k,
\end{equation}
so as not to deviate too much from gradient projection methods, while in our
case the perturbations are assumed to be just bounded.

Bounded perturbations as in (\ref{eq:summable-perturbations-introduction})
were previously studied in the context of inexact matrix splitting algorithms
for the symmetric monotone linear complementarity problem
\cite{Mangasarian1991Convergence}. This was further investigated in
\cite{Li1993Remarks} under milder assumptions by extending the proof of
\cite{Luo1992linear}. Additionally, convergence of the feasible descent method
with nonvanishing perturbations and its generalization to incremental
subgradient-type methods were also reported in \cite{Solodov1997Convergence}
and \cite{Solodov1998Error}, respectively.

The paper is organized as follows. In Section \ref{sect:PSG}, we introduce the
PSG methods by studying two particular cases of the proximity function
minimization problems for image reconstruction. In Section
\ref{sect:Convergence-PSG}, we present our main convergence results for the
PSG method with bounded perturbations, namely, the convergence of
(\ref{eq:PSG-with-summable-perturbations-introduction}%
)--(\ref{eq:summable-perturbations-introduction}). We call the latter
\textquotedblleft outer perturbations\textquotedblright\ because of the
location of the term $e(x^{k})$ in
(\ref{eq:PSG-with-summable-perturbations-introduction}). In Section
\ref{sect:BPR-PSG}, we prove the bounded perturbation resilience of the PSG
method by establishing a relationship between the inner perturbations and the
outer perturbations.


\section{Projected Scaled Gradient Methods\label{sect:PSG}}

In this section, we introduce the background and motivation of the projected
scaled gradient (PSG) methods for (\ref{eq:convex-minimization}). As mentioned
before, the PSG methods generate iterates according to the formula
\begin{equation}
\label{eq:PSG}x^{k+1} = P_{\Omega}(x^{k} -\tau_{k} D(x^{k}) \nabla
J(x^{k})),\quad k=0,1,2,\ldots
\end{equation}
where $\{\tau_{k}\}_{k=0}^{\infty}$ is a sequence of positive stepsizes and
$\{D(x^{k})\}_{k=0}^{\infty}$ is a sequence of diagonal scaling matrices. The
diagonal scaling matrices not only play the role of preconditioning the
gradient direction, but also induce a general algorithmic structure that
encompasses many existing algorithms as special cases.

In particular, the PSG methods include the gradient projection method
\cite{Bertsekas1976Goldstein,Goldstein1964Convex,Levitin1966Constrained},
which corresponds to the situation when $D(x^{k}) \equiv I_{n}$ for any $k$
with $I_{n}$ the identity matrix of order $n$. In case when $D(x^{k})
\approx\nabla^{2} J(x^{k})^{-1}$, namely when the diagonal scaling matrix is
an adequate approximation of the inverse Hessian, the PSG method reduces to
the projected Newton method \cite{Bertsekas1982Projected}. In fact, the
selection of various diagonal scaling matrices give rise to different concrete
algorithms. How to choose appropriate diagonal scaling matrices depends on the
particular problem.


We investigate the class of projected scaled gradient (PSG) methods by
concentrating on two particular cases of (\ref{eq:convex-minimization}).
Consider the following linear image reconstruction problem model with
nonnegativity constraint,
\begin{equation}
\label{eq:linear-problem}Ax=b,\text{ }x\geq0,
\end{equation}
where $A=(a_{j}^{i})_{i,j=1}^{m,n}$ is an $m\times n$ matrix in which
$a^{i}=(a_{j}^{i})_{j=1}^{n}\in\mathbb{R}^{n}$ is the $i$th column of its
transpose $A^{T}$, and $x=(x_{j})_{j=1}^{n}\in\mathbb{R}^{n}$ and
$b=(b_{i})_{i=1}^{m}\in\mathbb{R}^{m}$ are all assumed to be nonnegative. For
simplicity, we denote $\Omega_{0} := \mathbb{R}^{n}_{+} $ hereafter.

\subsection{Projected Landweber-type Methods}

The linear problem model (\ref{eq:linear-problem}) can be approached as the
following constrained weighted Least-Squares (LS) problem,
\begin{equation}
\label{eq:WLS-minimization}\left\{
\begin{array}
[c]{ll}%
\operatorname*{minimize} & \displaystyle J_{\mathrm{LS}}(x)\\
\text{subject to} & \displaystyle x \in\Omega_{0},
\end{array}
\right.
\end{equation}
where the weighted LS functional $J_{\mathrm{LS}}(x)$ is defined by
\begin{equation}
\label{eq:Weighted-LS}J_{\mathrm{LS}}(x):= \frac{1}{2}\left\|  b -Ax \right\|
_{W}^{2} = \frac{1}{2}\left\langle W(b -Ax) , b -Ax \right\rangle ,
\end{equation}
with $W$ the weighting matrix depending on the specific problem. The gradient
of $J_{\mathrm{LS}}(x)$ for any $x\in\mathbb{R}^{n}$ is
\begin{equation}
\label{eq:grad-WLS}\nabla J_{\mathrm{LS}}(x)= - A^{T} W (b - Ax).
\end{equation}

The projected Landweber method \cite[Section 6.2]{Bertero1998Introduction} for
(\ref{eq:WLS-minimization}) uses the iteration
\begin{equation}
\label{eq:projected-Landweber}x^{k+1} = P_{\Omega_{0}}(x^{k} + \tau_{k} A^{T}
W( b - A x^{k}) ).
\end{equation}
By (\ref{eq:grad-WLS}), the above (\ref{eq:projected-Landweber}) can be
written as
\begin{equation}
\label{eq:projected-Landweber-grad}x^{k+1} = P_{\Omega_{0}}(x^{k} - \tau_{k}
\nabla J_{\mathrm{LS}}(x^{k}) ),
\end{equation}
which obviously belongs to the family of PSG methods for
(\ref{eq:WLS-minimization}) with the diagonal scaling matrix $D(x^{k}) \equiv
I_{n}$ for any $k$.

The projected Landweber method with diagonal preconditioning for
(\ref{eq:WLS-minimization}), as studied in \cite{Piana1997Projected}, uses the
iteration
\begin{equation}
x^{k+1}=P_{\Omega_{0}}(x^{k}+\tau_{k}VA^{T}W(b-Ax^{k})),
\label{eq:preconditioned-projected-Landweber}%
\end{equation}
where $V$ is a diagonal $n\times n$ matrix satisfying certain conditions, see
\cite[p. 446, (i)-(iii)]{Piana1997Projected}. By (\ref{eq:grad-WLS}),
(\ref{eq:preconditioned-projected-Landweber}) is equivalent to the iteration
\begin{equation}
x^{k+1}=P_{\Omega_{0}}(x^{k}-\tau_{k}V\nabla J_{\mathrm{LS}}(x^{k})),
\label{eq:preconditioned-projected-Landweber-PSG}%
\end{equation}
and hence, it also belongs to the family of PSG methods with $D(x^{k})\equiv
V$ for any $k$.

In general, the projected Landweber-type methods for
(\ref{eq:WLS-minimization}) is given by
\begin{equation}
\label{eq:general-projected-Landweber}x^{k+1} = P_{\Omega_{0}}(x^{k} -
\tau_{k} D_{\mathrm{LS}} \nabla J_{\mathrm{LS}}(x^{k}) ),
\end{equation}
where the diagonal scaling matrices are typically constant positive definite
matrices of the form,
\begin{equation}
\label{eq:LS-scaling-matrix}D_{\mathrm{LS}}:=\operatorname*{diag}\left\{
\frac{1}{s_{j}} \right\}  , \quad s_{j} \in\mathbb{R} \text{~and~} s_{j} >0,
\text{~for all~} j=1,2,\ldots,n,
\end{equation}
with $s_{j}$ possibly constructed from the linear system matrix $A$ of
(\ref{eq:linear-problem}) for each $j$, and being sparsity pattern oriented
\cite[Eq. (2.2)]{Censor2008diagonally}.

\subsection{Generalized EM-type Methods}

The Kullback-Leibler distance is a widely adopted proximity function in the
field of image reconstruction. Using it, we seek a solution of
(\ref{eq:linear-problem}) by minimizing the Kullback-Leibler distance between
$b$ and $Ax$, as given by
\begin{equation}
\label{eq:jkl}J_{\mathrm{KL}}(x):= \mathrm{KL}(b,Ax) = \sum_{i=1}^{m}\left(
b_{i}\log\frac{b_{i}}{\left\langle a^{i},x\right\rangle }+\left\langle
a^{i},x\right\rangle -b_{i}\right)  ,
\end{equation}
over nonnegativity constraints, i.e.,%
\begin{equation}
\label{eq:KL-minimization}\left\{
\begin{array}
[c]{ll}%
\operatorname*{minimize} & \displaystyle J_{\mathrm{KL}}(x)\\
\text{subject to} & \displaystyle x \in\Omega_{0}.
\end{array}
\right.
\end{equation}
The gradient of $J_{\mathrm{KL}}(x)$ is
\begin{equation}
\label{eq:KL-gradient}\nabla J_{\mathrm{KL}}(x)=\sum_{i=1}^{m}\left(
1-\frac{b_{i}}{\left\langle a^{i},x\right\rangle }\right)  a^{i}.
\end{equation}

The class of EM-type algorithms is known to be closely related to KL
minimization. The $k$th iterative step of the EM algorithm in $\mathbb{R}^{n}$
is given by%
\begin{equation}
\label{eq:EM}x_{j}^{k+1}=\frac{x_{j}^{k}}{\sum_{i=1}^{m}a_{j}^{i}}\sum
_{i=1}^{m}\frac{b_{i}}{\left\langle a^{i},x^{k}\right\rangle }a_{j}^{i},\text{
for all }j=1,2,\ldots,n.
\end{equation}
The following convergence results of the EM algorithm are well-known. For any
positive initial point $x^{0}\in\mathbb{R}_{++}^{n}$, any sequence
$\{x^{k}\}_{k=0}^{\infty}$, generated by (\ref{eq:EM}), converges to a
solution of (\ref{eq:linear-problem}) in the consistent case, while it
converges to the minimizer of the Kullback-Leibler distance $\mathrm{KL}%
(b,Ax)$, defined by (\ref{eq:jkl}), in the inconsistent case
\cite{Iusem1991Convergence}.

It is known that the EM algorithm can be viewed as the following scaled
gradient method, see, e.g.,
\cite{Bertero2008Iterative,HelouNeto2005Convergence,Lanteri2001general}, whose
$k$th iterative step is
\begin{equation}
x^{k+1}=x^{k}-D_{\mathrm{EM}}(x^{k})\nabla J_{\mathrm{KL}}(x^{k}),
\label{equ:EM-scaled-gradient}%
\end{equation}
where the $n\times n$ diagonal scaling matrix is defined by
\begin{equation}
\label{eq:EM-scaling-matrix}D_{\mathrm{EM}}(x):=\operatorname*{diag}\left\{
\frac{x_{j}}{\sum_{i=1}^{m}a_{j}^{i}}\right\}  .
\end{equation}
Thus the EM algorithm belongs to the class of PSG methods with $\tau_{k}
\equiv1$ for all $k$ and the diagonal scaling matrix given by $D(x) \equiv
D_{\mathrm{EM}}(x)$ for any $x$.

More generally, generalized EM-type methods for (\ref{eq:KL-minimization}) can
be given by
\begin{equation}
\label{eq:generalized-EM}x^{k+1}=P_{\Omega_{0}}(x^{k}-\tau_{k}D_{\mathrm{KL}%
}(x^{k})\nabla J_{\mathrm{KL}}(x^{k})),
\end{equation}
with $\{\tau_{k}\}_{k=0}^{\infty}$ as relaxation parameters \cite[Section
5.1]{Censor1997Parallel} and $\{D_{\mathrm{KL}}(x^{k})\}_{k=0}^{\infty}$ as
diagonal scaling matrices. The diagonal scaling matrices for the generalized
EM-type methods are typically of the form, see, e.g.,
\cite{HelouNeto2005Convergence},
\begin{equation}
\label{eq:KL-scaling-matrix-general}D_{\mathrm{KL}}(x):=\operatorname*{diag}%
\left\{  \frac{x_{j}}{\hat{s}_{j}}\right\}  ,\quad\text{with~} \hat{s}_{j}
\in\mathbb{R} \text{~and~} \hat{s}_{j}>0 \text{~for~} j=1,2, \ldots,n,
\end{equation}
where $\hat{s}_{j}$ might be dependent on the linear system matrix $A$ of
(\ref{eq:linear-problem}) for any $j$. When $\hat{s}_{j} = \sum_{i=1}^{m}%
a_{j}^{i}$ for any $j$, then $D_{\mathrm{KL}}(x)$ coincides with the matrix
$D_{\mathrm{EM}}(x)$ given by (\ref{eq:EM-scaling-matrix}).

\bigskip

It is worthwhile to comment here that it is natural to obtain incremental
versions of PSG methods when the objective function $J(x)$ is separable, i.e.,
$J(x) = \sum_{i=1}^{m} J_{i}(x)$ for some integer $m$. The separability of
both the weighted LS functional (\ref{eq:Weighted-LS}) and the KL functional
(\ref{eq:jkl}) facilitates the derivation of incremental variants for the
projected Landweber-type methods and generalized EM-type methods. While the
incremental methods enjoy better convergence at early iterations, relaxation
strategies are required to guarantee asymptotic acceleration
\cite{HelouNeto2009Incremental}.

\section{Convergence of the PSG Method with Outer
Perturbations\label{sect:Convergence-PSG}}

In this section, we present our main convergence results of the PSG method
with bounded outer perturbations of the form
(\ref{eq:PSG-with-summable-perturbations-introduction}%
)--(\ref{eq:summable-perturbations-introduction}). The stationary points of
(\ref{eq:convex-minimization}) are fixed points of $P_{\Omega}(x-\nabla J(x))$
\cite[Corollary 1.3.5]{Cegielski2013Iterative}, i.e., zeros of the residual
function%
\begin{equation}
r(x):=x-P_{\Omega}(x-\nabla J(x)). \label{eq:residual-function}%
\end{equation}
We denote the set of all these stationary points by
\begin{equation}
S:=\left\{  x\in\mathbb{R}^{n}\mid r(x)=0\right\}  ,
\label{eq:stationary-points-set}%
\end{equation}
and assume that\ $S\neq\emptyset$. We also assume that (1) has a solution and
that $J^{\ast}:=\inf_{x\in\Omega}J(x).$ We will prove that sequences generated
by a PSG method converge to a stationary point of
(\ref{eq:convex-minimization}) in the presence of bounded perturbations.

We focus our attention on objective functions $J(x)$ of
(\ref{eq:convex-minimization}) that are assumed to belong to a subclass of
convex functions, in the notation of \cite[p. 65]{Nesterov2004Introductory},
$J \in\mathcal{S}_{\mu,L}^{1,1}(\Omega)$, which means that $\nabla J$ is
Lipschitz continuous on $\Omega$ with Lipschitz constant $L$, i.e., there
exists a $L>0$, such that%
\begin{equation}
\label{eq:Lipschitz-gradient}\Vert\nabla J(x)-\nabla J(y)\Vert\leq L\Vert
x-y\Vert,\quad\text{for all }x,y\in\Omega,
\end{equation}
and that $J$ is strongly convex on $\Omega$ with the strong convexity
parameter $\mu$ ($L \geq\mu$), i.e., there exists a $\mu>0$, such that%
\begin{equation}
\label{eq:stron-conv}J(y)\geq J(x)+\langle\nabla J(x),y-x\rangle+\frac{1}%
{2}\mu\Vert y-x\Vert^{2},\text{ \ for all }x,y\in\Omega.
\end{equation}
The convergence of gradient methods without perturbations for this subclass of
convex functions, $\mathcal{S}_{\mu,L}^{1,1}(\Omega)$, is well-established,
see \cite{Nesterov2004Introductory}.

Motivated by recent works on superiorization
\cite{Censor2010Perturbation,Censor2013Projected,Herman2012Superiorization}
and the framework of feasible descent methods \cite{Luo1993Error}, we
investigate convergence of the PSG method with bounded perturbations for
(\ref{eq:convex-minimization}), that is,
\begin{equation}
x^{k+1}=P_{\Omega}(x^{k}-\tau_{k}D(x^{k})\nabla J(x^{k})+e(x^{k})),
\label{eq:PSG-with-summable-perturbations}%
\end{equation}
where $\{\tau_{k}\}_{k=0}^{\infty}$ is a sequence of positive scalars with
\begin{equation}
\label{eq:stepsize-bound}0 < \inf_{k}\tau_{k} \leq\tau_{k} \leq\sup_{k}%
\tau_{k} < 2/L,
\end{equation}
and $\{D(x^{k})\}_{k=0}^{\infty}$ is a sequence of diagonal scaling matrices.
Denoting $e^{k}:=e(x^{k})$, the sequence of perturbations $\{e^{k}%
\}_{k=0}^{\infty}$ is assumed to be summable, i.e.,
\begin{equation}
\sum_{k=0}^{\infty}\Vert e^{k}\Vert<+\infty. \label{eq:summable-perturbations}%
\end{equation}
To ensure that the scaled gradient direction does not deviate too much from
the gradient direction, we define%
\begin{equation}
\theta^{k}:=\nabla J(x^{k})-D(x^{k})\nabla J(x^{k}), \label{eq:theta-k}%
\end{equation}
and assume that
\begin{equation}
\sum_{k=0}^{\infty}\Vert\theta^{k}\Vert<+\infty. \label{eq:grad-dev-summable}%
\end{equation}

\subsection{Preliminary Results}

In this subsection, we prepare some relevant facts and pertinent conditions
that are necessary for our convergence analysis. The following lemmas are
required by subsequent proofs. The first one is known as the descent lemma for
a function with Lipschitz continuous gradient, see \cite[Proposition
A.24]{Bertsekas1999Nonlinear}.

\begin{lemma}
\label{lemma:Lipschitz}Let $J:\mathbb{R}^{n}\rightarrow\mathbb{R}$ be a
continuously differentiable function whose gradients are Lipschitz continuous
with constant $L$. Then, for any $L^\prime\geq L$,
\begin{equation}
J(x)\leq J(y)+\langle\nabla J(y),x-y\rangle+\frac{L^\prime}{2}\Vert
x-y\Vert^{2},\quad\text{for all }x,y\in\mathbb{R}^{n}.
\end{equation}
\end{lemma}

The second lemma reveals well-known characterizations of projections onto
convex sets, see, e.g., \cite[Proposition 2.1.3]{Bertsekas1999Nonlinear} or
\cite[Fig. 11]{Polyak1987Introduction}.

\begin{lemma}\label{lemma:projection-onto-convex-sets}
Let $\Omega$ be a nonempty, closed and convex subset of $\mathbb{R}^n$. Then, the orthogonal projection onto $\Omega$ is characterized by
\begin{description}
\item[(i)]
For any $x \in \mathbb{R}^n$, the projection $P_{\Omega}(x)$ of $x$ onto $\Omega$ satisfies
\begin{equation}\label{eq:projection-onto-convex-sets}
\left \langle x - P_{\Omega}(x), y - P_{\Omega}(x)\right \rangle \leq 0, \quad \forall y \in \Omega.
\end{equation}
\item[(ii)]
$P_{\Omega}$ is a nonexpansive operator, i.e.,
\begin{equation}\label{eq:projection-nonexpansive}
\|P_{\Omega}(x) - P_{\Omega}(y)\| \leq \| x - y\|, \quad \forall x,y \in \mathbb{R}^n.
\end{equation}
\end{description}
\end{lemma}

The third lemma is a property of the orthogonal projection operator, which was
proposed in \cite[Lemma 1]{Gafni1984Two}, see also \cite[Lemma 2.3.1]%
{Bertsekas1999Nonlinear}.

\begin{lemma}\label{lemma:geometry-projection}
Let $\Omega$ be a nonempty, closed and convex subset of $\mathbb{R}^n$. Given $x\in\mathbb{R}^{n}$ and $d\in
\mathbb{R}^{n}$, the function $\varphi(t)$ defined by
\begin{equation}
\varphi(t):=\frac{\left\Vert P_{\Omega}(x+td)-x\right\Vert }{t}%
\end{equation}
is monotonically nonincreasing for $t>0$.
\end{lemma}

The fourth lemma is from \cite[Lemma 2.2]{Mangasarian1991Convergence}, which originates from \cite[Lemma 2.1]{Cheng1984gradient}, see also \cite[Lemma 3.1]{Combettes2001Quasi} or \cite[p. 44, Lemma 2]{Polyak1987Introduction} for a more general formulation.
\begin{lemma}\label{lemma:Cheng}
Let $\{\alpha_{k}\}_{k=0}^{\infty} \subset \mathbb{R}_+$ be a sequence of nonnegative real numbers. If it holds that $0 \leq \alpha_{k+1} \leq \alpha_k + \varepsilon_{k}$ for all $k \geq 0$, where $\varepsilon_{k}\geq0$ for all $k\geq0$ and $\sum_{k=0}^{\infty}\varepsilon_{k}<+\infty$,
then the sequence $\{\alpha_{k}\}_{k=0}^{\infty}$ converges.
\end{lemma}

In our analysis we make use of the following two conditions, which are
Assumptions A and B, respectively, in \cite{Luo1993Error}, and are called
\textquotedblleft local error bound\textquotedblright\ condition and
\textquotedblleft proper separation of isocost surfaces\textquotedblright%
~condition, respectively. The error bound condition estimates the distance of
an $x\in\Omega$ to the solution set $S$, defined above, by the norm of the
residual function, see \cite{Pang1997Error} for a comprehensive review. Denote
the distance from a point $x$ to the set $S$ by $d(x,S)=\min_{y\in S}\Vert
x-y\Vert$.

\begin{condition}
\label{assumption:error-bound} For every $v\geq\inf_{x\in\Omega}J(x)$, there
exist scalars $\varepsilon>0$ and $\beta>0$ such that
\begin{equation}
\label{eq:local-error-bound}d(x,S)\leq\beta\Vert r(x)\Vert
\end{equation}
for all $x\in\Omega$ with $J(x)\leq v$ and $\Vert r(x)\Vert\leq\varepsilon$.
\end{condition}

The second condition, which says that the isocost surfaces of the function
$J(x)$ on the solution set $S$ should be properly separated, is known to hold
for any convex function \cite[p. 161]{Luo1993Error}.

\begin{condition}
\label{assumption:proper-separation} There exists a scalar $\varepsilon>0$
such that
\begin{equation}
\text{if }u,v\in S\text{ and }J(u)\not =J(v)\text{ then }\Vert u-v\Vert
\geq\varepsilon.
\end{equation}

\end{condition}

Next, we show that the above two conditions are satisfied by functions
belonging to $\mathcal{S}_{\mu,L}^{1,1}(\Omega)$. Since Condition
\ref{assumption:proper-separation} certainly holds for a strongly convex
function, we need to prove that Condition \ref{assumption:error-bound} is also
fulfilled. The early roots of the proof of the next lemma, which leads to this fact, can be traced back to Theorem 3.1 of \cite{Pang1987posteriori}.

\begin{lemma}\label{lemma:strongly-convex-error-bound-hold}
The error bound condition (\ref{eq:local-error-bound}) holds globally for any $J \in \mathcal{S}_{\mu,L}^{1,1}(\Omega)$.
\end{lemma}

\begin{proof}
By the definition of the residual function (\ref{eq:residual-function}), we have
\begin{equation}\label{eq:x-plus-rx}
x - r(x) = P_{\Omega}(x - \nabla J(x)) \in \Omega.
\end{equation}
For any given $x^*\in S$, by the optimality condition of the problem (\ref{eq:convex-minimization}), see, e.g., \cite[p. 203, Theorem 3]{Polyak1987Introduction} or \cite[Proposition 2.1.2]{Bertsekas1999Nonlinear}, we know that
\begin{equation}\label{eq:optimility-condition}
\langle \nabla J(x^*), x - x^* \rangle \geq 0, \quad \forall x \in \Omega.
\end{equation}
Since $x - r(x) \in \Omega$ for all $x \in \Omega$, then, by (\ref{eq:optimility-condition}), we obtain,
\begin{equation}\label{equ:optimality-10}
\langle - \nabla J(x^*), x - r(x) - x^* \rangle \leq 0.
\end{equation}
From Lemma \ref{lemma:projection-onto-convex-sets} (i) and (\ref{eq:x-plus-rx}), we get
\begin{align}
&\quad~ \left\langle (x - \nabla J(x)) - P_{\Omega}(x - \nabla J(x)),  x^* - P_{\Omega}(x - \nabla J(x)) \right\rangle \leq 0 \nonumber\\
&\Rightarrow\langle (x - \nabla J(x)) - (x-r(x)), x^* - (x-r(x)) \rangle \leq 0 \nonumber\\
&\Rightarrow\langle \nabla J(x) - r(x), x-r(x) - x^* \rangle \leq 0 \nonumber\\
&\Rightarrow\langle \nabla J(x), x-r(x) - x^* \rangle \leq \langle r(x), x-r(x) - x^* \rangle.
\label{equ:convexity-10}
\end{align}
Summing up both sides of (\ref{equ:optimality-10}) and (\ref{equ:convexity-10}), yields
\begin{align}\label{eq:gradjx-plus-gradjxstar-upper-bound}
&\quad~ \left\langle \nabla J(x) - \nabla J(x^*), x - r(x) -x^* \right\rangle \leq \langle r(x), x-r(x)-x^* \rangle \nonumber\\
&\Rightarrow\langle \nabla J(x) - \nabla J(x^*), x - x^* \rangle \leq
\langle r(x), \nabla J(x) - \nabla J(x^*) + x - x^* \rangle.
\end{align}
By the strong convexity of $J(x)$, we have that \cite[Theorem 2.1.9 ]{Nesterov2004Introductory},
\begin{equation}\label{eq:strongly-convex-property}
\langle \nabla J(x) - \nabla J(x^*), x - x^* \rangle \geq \mu\|x - x^*\|^2.
\end{equation}
Combing (\ref{eq:gradjx-plus-gradjxstar-upper-bound}) with (\ref{eq:strongly-convex-property}), leads to
\begin{align}
\mu \|x -x^* \|^2 &\leq \langle r(x),\nabla J(x) -\nabla J(x^*) +x -x^* \rangle \nonumber\\
&\leq (\|\nabla J(x) - \nabla J(x^*) \|+\|x- x^* \|) \|r(x) \| \nonumber\\
&\leq (L+1)\|x -x^* \| \|r(x) \| \nonumber\\
\Rightarrow~   \|x -x^*\| &\leq (L+1)/\mu ~\|r(x) \|,
\end{align}
and, hence,
\begin{equation}
d(x, S) \leq (L+1)/\mu ~\|r(x) \|.
\end{equation}
Consequently, the error bound condition (\ref{eq:local-error-bound}), namely Condition \ref{assumption:error-bound} holds.
\end{proof}

\subsection{Convergence Analysis}

In this subsection, we give the detailed convergence analysis for the PSG
method with bounded outer perturbations of
(\ref{eq:PSG-with-summable-perturbations}). The proof techniques follow the
track of
\cite{Li1993Remarks,Luo1992linear,Luo1993Error,Mangasarian1991Convergence} and
extend them to adapt to our case here. We first prove the convergence of the
sequence of objective function values $\{J(x^{k})\}_{k=0}^{\infty}$ at points
of any sequence $\{x^{k}\}_{k=0}^{\infty}$ generated by the PSG method with
bounded outer perturbations of (\ref{eq:PSG-with-summable-perturbations}). We
then prove that any sequence of points $\{x^{k}\}_{k=0}^{\infty}$, generated
by the PSG method with bounded outer perturbations of
(\ref{eq:PSG-with-summable-perturbations}), converges to a stationary point.

The following proposition estimates the difference of objective function
values between successive iterations in the presence of bounded perturbations.

\begin{proposition}
\label{proposition:descent-condition-bound}Let $\Omega\subseteq\mathbb{R}^{n}$ be a
nonempty closed convex set and assume that $J(x)$ is strongly convex on $\Omega$
with convexity parameter $\mu,$ and that $\nabla J$ is Lipschitz continuous on
$\Omega$ with Lipschitz constant $L$ such that $L\geq\mu$. Further, let
$\{\tau_{k}\}_{k=0}^{\infty}$ be a sequence of positive scalars that fulfills
(\ref{eq:stepsize-bound}), let $\{e^{k}\}_{k=0}^{\infty}$ be a sequence of
perturbation vectors as defined above that fulfills
(\ref{eq:summable-perturbations}), and let $\{\theta^{k}\}_{k=0}%
^{\infty}$ be as in (\ref{eq:theta-k}) and for which
(\ref{eq:grad-dev-summable}) holds. If $\{x^{k}%
\}_{k=0}^{\infty}$ is any sequence, generated by the PSG method with bounded outer perturbations of (\ref{eq:PSG-with-summable-perturbations}), then there exists an $\eta_{1}>0$
such that%
\begin{equation}
J(x^{k})-J(x^{k+1})\geq\eta_{1}\Vert x^{k}-x^{k+1}\Vert^{2}-\Vert\delta
^{k}\Vert\Vert x^{k}-x^{k+1}\Vert\label{equ:descent-condition-bound}%
\end{equation}
with $\delta^{k}$ defined via the above-mentioned $\tau_{k},$ $e^{k}$ and $\theta^{k}$, by
\begin{equation}\label{eq:error-sum-delta}%
\delta^{k}:=\frac{1}{\tau_{k}}e^{k}+\theta^{k}.
\end{equation}
\end{proposition}

\begin{proof}
Lemma \ref{lemma:Lipschitz} implies that
\begin{equation}
J(x^{k})-J(x^{k+1})\geq\langle\nabla J(x^{k}),x^{k}-x^{k+1}\rangle-\frac{L}%
{2}\Vert x^{k}-x^{k+1}\Vert^{2}. \label{equ:descent-by-Lipschitz}%
\end{equation}
By (\ref{eq:PSG-with-summable-perturbations}) and Lemma \ref{lemma:projection-onto-convex-sets}, we have%
\begin{equation}
\langle x^{k+1}-x^{k},x^{k}-\tau_{k}D(x^{k})\nabla J(x^{k})+e^{k}%
-x^{k+1}\rangle\geq0.
\end{equation}
Rearrangement of the last relation and using (\ref{eq:theta-k}) leads to%
\begin{align}
\langle\nabla J(x^{k}),x^{k}-x^{k+1}\rangle
&\geq \frac{1}{\tau_{k}}\Vert
x^{k}-x^{k+1}\Vert^{2}+\frac{1}{\tau_{k}}\langle e^{k},x^{k}-x^{k+1}%
\rangle\nonumber\\
&\quad~  +\langle\theta^{k},x^{k}-x^{k+1}\rangle.
\end{align}
By (\ref{eq:error-sum-delta}) and the Cauchy-Schwarz inequality we then
obtain
\begin{equation}
\langle\nabla J(x^{k}),x^{k}-x^{k+1}\rangle\geq\frac{1}{\tau_{k}}\Vert
x^{k}-x^{k+1}\Vert^{2}-\Vert\delta^{k}\Vert\Vert x^{k}-x^{k+1}\Vert.
\label{equ:descent-estimation-lowerbound}%
\end{equation}
Combining (\ref{equ:descent-estimation-lowerbound}) with
(\ref{equ:descent-by-Lipschitz}) leads to%
\begin{equation}
J(x^{k})-J(x^{k+1})\geq(\frac{1}{\tau_{k}}-\frac{L}{2})\Vert x^{k}%
-x^{k+1}\Vert^{2}-\Vert\delta^{k}\Vert\Vert x^{k}-x^{k+1}\Vert.
\end{equation}
By defining $\displaystyle \overline{\tau}:=\sup_{k}\tau_{k}$ and
\begin{equation}
\eta_{1}:=\frac{1}{\overline{\tau}}-\frac{L}{2}, \label{equ:eta2}%
\end{equation}
the proof is complete.
\end{proof}

From Proposition \ref{proposition:descent-condition-bound} and Lemma
\ref{lemma:Cheng}, we obtain the following theorem on the convergence of
objective function values.

\begin{theorem}
\label{theorem:f-value-converge}If the problem
(\ref{eq:convex-minimization}) has a solution, namely
$J^{\ast}=\inf_{x\in \Omega}J(x)$,
then under the conditions of Proposition
\ref{proposition:descent-condition-bound}, the sequence of function values
$\{J(x^{k})\}_{k=0}^{\infty}$ calculated at points of any sequence
$\{x^{k}\}_{k=0}^{\infty},$ generated by the PSG method with bounded outer
perturbations of (\ref{eq:PSG-with-summable-perturbations}), converges.
\end{theorem}

\begin{proof}
From Proposition \ref{proposition:descent-condition-bound}, we can further get%
\begin{equation}
J(x^{k})-J(x^{k+1})\geq\eta_{1}\left(  \Vert x^{k}-x^{k+1}\Vert-\frac{1}%
{2\eta_{1}}\Vert\delta^{k}\Vert\right)  ^{2}-\frac{1}{4\eta_{1}}\Vert
\delta^{k}\Vert^{2}, \label{equ:descent-condition-bound-others}%
\end{equation}
and since $J(x)\geq J^{\ast}$, for all $x\in \Omega,$ the above relation implies
that%
\begin{equation}
0\leq J(x^{k+1})-J^{\ast}\leq J(x^{k})-J^{\ast}+\frac{1}{4\eta_{1}}\Vert
\delta^{k}\Vert^{2}. \label{equ:fvalue-minus-fmin}%
\end{equation}
By defining $\displaystyle \underline{\tau}:=\inf_{k}\tau_{k}$ and using
Minkowski's inequality, we get%
\begin{equation}
\Vert\delta^{k}\Vert^{2}\leq\frac{1}{\tau_{k}^{2}}\Vert e^{k}\Vert^{2}%
+\Vert\theta^{k}\Vert^{2}\leq\frac{1}{\underline{\tau}^{2}}\Vert e^{k}%
\Vert^{2}+\Vert\theta^{k}\Vert^{2}, \label{eq:error-sum-delta-upperbound}%
\end{equation}
which implies, by (\ref{eq:summable-perturbations}) and
(\ref{eq:grad-dev-summable}), that $\sum_{k=1}^{\infty} \Vert\delta^{k}\Vert^{2}<+\infty$.
Then, by Lemma \ref{lemma:Cheng} and (\ref{equ:fvalue-minus-fmin}), the sequence $\{J(x^k)- J^{\ast}\}_{k=0}^\infty$ converges, and hence the sequence $\{J(x^k)\}_{k=0}^\infty$ also converges.
\end{proof}

In what follows, we prove that any sequence, generated by the PSG method with
bounded outer perturbations of (\ref{eq:PSG-with-summable-perturbations}),
converges to a stationary point of $S$. The following propositions lead to
that result. The first proposition shows that $\Vert x^{k}-x^{k+1}\Vert$ is
bounded above by the difference between objective function values at
corresponding points plus a perturbation term.

\begin{proposition}
\label{proposition:sequence-error-estimation}Under the conditions of
Proposition \ref{proposition:descent-condition-bound}, let $\{x^{k}%
\}_{k=0}^{\infty}$ be any sequence generated by the PSG method with bounded outer perturbations of (\ref{eq:PSG-with-summable-perturbations}). Let $\eta_{1}$ be given by
(\ref{equ:eta2}) and $\{\delta^{k}\}_{k=0}^{\infty}$ be given by
(\ref{eq:error-sum-delta}). Then, it holds that
\begin{equation}\label{equ:sequence-error-upperbound}%
\Vert x^{k}-x^{k+1}\Vert\leq\sqrt{\frac{2}{\eta_{1}}}\left\vert J(x^{k}%
)-J(x^{k+1})\right\vert ^{1/2}+\frac{1}{\eta_{1}}\Vert\delta^{k}
\Vert.
\end{equation}
\end{proposition}

\begin{proof}
By the basic inequality $(p+q)^{2}\leq2(p^{2}+q^{2}), \forall p,q\in\mathbb{R}$, we can write%
\begin{equation}
\Vert x^{k}-x^{k+1}\Vert^{2}\leq2\left(  \left(  \Vert x^{k}-x^{k+1}%
\Vert-\frac{1}{2\eta_{1}}\Vert\delta^{k}\Vert\right)  ^{2}+\left(  \frac
{1}{2\eta_{1}}\Vert\delta^{k}\Vert\right)  ^{2}\right)  . \label{eq:ppqq}%
\end{equation}
From (\ref{equ:descent-condition-bound-others}) and (\ref{eq:ppqq}), we
have%
\begin{equation}
\Vert x^{k}-x^{k+1}\Vert^{2}\leq\frac{2}{\eta_{1}}\left(  J(x^{k}%
)-J(x^{k+1})\right)  +\frac{1}{\eta_{1}^{2}}\Vert\delta^{k}\Vert^{2},
\label{equ:sequence-error-square-upperbound}%
\end{equation}
which allows us to use the inequality $\sqrt{a^{2}+b^{2}}\leq a+b,\forall a,b\geq0$, yielding (\ref{equ:sequence-error-upperbound}).
\end{proof}

The next proposition gives an upper bound on the residual function of
(\ref{eq:residual-function}) in the presence of bounded perturbations.

\begin{proposition}
\label{proposition:residual-upper-bound}Under the conditions of Proposition
\ref{proposition:descent-condition-bound}, if $\{x^{k}\}_{k=0}^{\infty}$ is
any sequence generated by the PSG method with bounded outer perturbations of
(\ref{eq:PSG-with-summable-perturbations}), then there exists a constant $\eta_{2}>0$ such
that, for the residual function of (\ref{eq:residual-function}) we have, for
all $k\geq0,$%
\begin{equation}
\Vert r(x^{k})\Vert\leq\eta_{2}(\Vert x^{k}-x^{k+1}\Vert+\Vert e^{k}%
\Vert+\Vert\theta^{k}\Vert). \label{equ:residual-upper-bound}%
\end{equation}
\end{proposition}

\begin{proof}
From (\ref{eq:PSG-with-summable-perturbations}), it holds true, by (\ref{eq:projection-nonexpansive}), that
\begin{equation}
\Vert x^{k+1}-P_{\Omega}(x^{k}-\tau_{k}D(x^{k})\nabla J(x^{k}))\Vert \leq \Vert e^{k}\Vert.
\end{equation}
Then, we can get
\begin{align}
&\quad~\Vert x^{k}-P_{\Omega}(x^{k}-\tau_{k}D(x^{k})\nabla J(x^{k}))\Vert \nonumber\\
&\leq\Vert x^{k}-x^{k+1}\Vert +\Vert x^{k+1}-P_{\Omega}(x^{k}-\tau_{k}D(x^{k})\nabla J(x^{k}))\Vert \nonumber\\
&\leq\Vert x^{k}-x^{k+1}\Vert+\Vert e^{k}\Vert. \label{eq:res-upperbound-1}%
\end{align}
By Lemma \ref{lemma:geometry-projection}, the left-hand side of (\ref{eq:res-upperbound-1}) is bounded below, according to
\begin{equation}\label{eq:res-lowerbound-1}
\Vert x^{k}-P_{\Omega}(x^{k}-\tau_{k}D(x^{k})\nabla J(x^{k}))\Vert \geq \hat{\tau}\Vert x^{k}-P_{\Omega}(x^{k}-D(x^{k})\nabla J(x^{k}))\Vert
\end{equation}
with $\displaystyle \hat{\tau}:=\min\{1,\inf_{k}\tau_{k}\}>0$.
By (\ref{eq:res-upperbound-1}) and (\ref{eq:res-lowerbound-1}), we then obtain
\begin{equation}
\Vert x^{k}-P_{\Omega}(x^{k}-D(x^{k})\nabla J(x^{k}))\Vert\leq\frac{1}{\hat{\tau}%
}(\Vert x^{k}-x^{k+1}\Vert+\Vert e^{k}\Vert).
\label{equ:scaled-error-bound-intermediate}%
\end{equation}
By the nonexpansiveness of the projection operator (\ref{eq:projection-nonexpansive}), and the triangle inequality, we see that the residual function, defined by (\ref{eq:residual-function}), satisfies%
\begin{align}
\Vert r(x^{k})\Vert &\leq  \Vert x^{k}-P_{\Omega}(x^{k}-D(x^{k})\nabla
J(x^{k}))\Vert \nonumber\\
&\quad + \Vert P_{\Omega}(x^{k}-D(x^{k})\nabla J(x^{k}))-P_{\Omega}(x^{k}-\nabla
J(x^{k}))\Vert\nonumber\\
&\leq\Vert x^{k}-P_{\Omega}(x^{k}-D(x^{k})\nabla J(x^{k}))\Vert+\Vert\nabla
J(x^{k})-D(x^{k})\nabla J(x^{k})\Vert\nonumber\\
&\leq\frac{1}{\hat{\tau}}(\Vert x^{k}-x^{k+1}\Vert+\Vert e^{k}\Vert
)+\Vert\theta^{k}\Vert),
\end{align}
which, by choosing $\eta_{2}:=$ $\displaystyle\frac{1}{\hat{\tau}},$ completes
the proof.
\end{proof}

The next proposition estimates the difference between the objective function
value at the current iterate and the optimal value. The proof is inspired by
that of \cite[Theorem 3.1]{Luo1993Error}.

\begin{proposition}
\label{proposition:f-and-fmin-error-upperbound}Under the conditions of
Proposition \ref{proposition:descent-condition-bound}, if $\{x^{k}%
\}_{k=0}^{\infty}$ is any sequence generated by the PSG method with bounded outer
perturbations of (\ref{eq:PSG-with-summable-perturbations}), then there exists a constant
$\eta_{3}>0$ and an index $ K_{3}>0 $ such that for all $ k > K_{3} $%
\begin{equation}\label{equ:f-and-fmin-error-upperbound}%
J(x^{k+1})-J^{\ast}\leq\eta_{3}\left(  \Vert x^{k}-x^{k+1}\Vert+\Vert
e^{k}\Vert + \Vert\theta^{k}\Vert\right)  ^{2}.
\end{equation}
\end{proposition}

\begin{proof}
Note that (\ref{eq:summable-perturbations}) and (\ref{eq:grad-dev-summable}) imply that $\lim_{k\rightarrow \infty}\Vert e^{k}\Vert=0$ and $\lim_{k\rightarrow\infty}\Vert\theta^{k}%
\Vert=0$, respectively, hence, $\lim_{k\rightarrow\infty}\Vert\delta^{k}%
\Vert=0$. Then, Theorem \ref{theorem:f-value-converge} and Proposition
\ref{proposition:sequence-error-estimation} imply that%
\begin{equation}\label{eq:lim-cons}%
\lim_{k\rightarrow\infty}\Vert x^{k}-x^{k+1}\Vert=0,
\end{equation}
and Proposition \ref{proposition:residual-upper-bound} shows that%
\begin{equation}
\lim_{k\rightarrow\infty}\Vert r(x^{k})\Vert=0.
\label{equ:residual-limit-zero}%
\end{equation}
Condition \ref{assumption:error-bound} guarantees that there exist an index
$K_{2}>K_{1}$ and a scalar $\beta>0$ such that for all $k>K_{2}$%
\begin{equation}
\Vert x^{k}-\hat{x}^{k}\Vert\leq\beta\Vert r(x^{k})\Vert,
\label{eq:local-error-bound-application}%
\end{equation}
where $\hat{x}^{k}\in S$ is a point for which $d(x^{k},S)=\Vert x^{k}-\hat
{x}^{k}\Vert$. The last two relations (\ref{equ:residual-limit-zero}) and
(\ref{eq:local-error-bound-application}) then imply that
\begin{equation}
\lim_{k\rightarrow\infty}(x^{k}-\hat{x}^{k})=0,
\label{equ:error-between-xn-and-solution-limit-zero}%
\end{equation}
and, using the triangle inequality and (\ref{eq:lim-cons}), we get%
\begin{equation}
\lim_{k\rightarrow\infty}(\hat{x}^{k}-\hat{x}^{k+1})=0. \label{eq:hats}%
\end{equation}
In view of Condition \ref{assumption:proper-separation}, and since $\hat
{x}^{k}\in S$ for all $k\geq0,$ (\ref{eq:hats}) implies that there exists an
integer $K_{3}>K_{2}$ and a scalar $J^{\infty}$ such that%
\begin{equation}
J(\hat{x}^{k})=J^{\infty},\quad\text{for all }k>K_{3}. \label{eq:Jinfinity}%
\end{equation}
Next we show that $J^{\infty}=J^{\ast}$. For any $k>K_{3}$, since $\hat{x}^{k}$ is a stationary
point of $J(x)$ over $\Omega$, it is true that
\begin{equation}
\langle \nabla J(\hat{x}^{k}), x-\hat{x}^{k} \rangle \geq 0, \quad \forall x\in\Omega .
\end{equation}
From the optimality condition of constrained convex optimization \cite[Proposition 2.1.2]{Bertsekas1999Nonlinear}, we obtain that
\begin{equation}
J(x) \geq J(\hat{x}^{k})=J^\infty, \quad \forall x\in\Omega.
\end{equation}
By the definition of $J^{\ast}$, we have $J(x) \geq J^\infty \geq J^{\ast}$ for any $x\in \Omega$, and hence
\begin{equation}\label{eq:JisJ}
J^\infty = J^{\ast}.
\end{equation}
If not, then $J^\infty > J^{\ast}$, which means that $J^\infty$ will be the infimum of $J(x)$ over $\Omega$
instead of $J^{\ast}$ and contradiction occurs.
\begin{figure}[htb]
\begin{center}
\includegraphics[height=4.4cm]{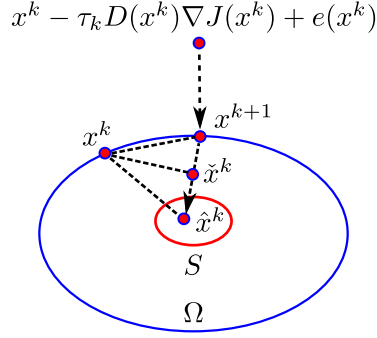}
\caption{An illustration of the geometric relationship between points $x^k$, $x^{k+1}$, $\check{x}^k$ and $\hat{x}^k$}
\label{fig:Convex-Projection}
\end{center}
\end{figure}
\newline\noindent Since $\Omega$ is convex and $x^{k+1}$ is the projection of $x^{k}-\tau
_{k}D(x^{k})\nabla J(x^{k})+e^{k}$ onto $\Omega$ (See Fig. \ref{fig:Convex-Projection}), by Lemma \ref{lemma:projection-onto-convex-sets} (i), the following inequality holds%
\begin{equation}
\langle x^{k}-\tau_{k}D(x^{k})\nabla J(x^{k})+e(x^{k})-x^{k+1},x^{k+1}-\hat
{x}^{k}\rangle\geq0,
\end{equation}
and arrangement of the terms leads to%
\begin{align}\label{eq:descent-estimation-lowerbound-others}%
&\quad~  \langle\nabla J(x^{k}),x^{k+1}-\hat{x}^{k}\rangle\nonumber\\
&  \leq \langle\theta^{k}+\frac{1}{\tau_{k}}e^{k},x^{k+1}-\hat{x}^{k}\rangle+\frac{1}{\tau_{k}}\langle x^{k}-x^{k+1},x^{k+1}-\hat{x}^{k}%
\rangle\nonumber\\
&  \leq\left(  \Vert\theta^{k}\Vert+\frac{1}{\underline{\tau}}\Vert e^{k}%
\Vert+\frac{1}{\underline{\tau}}\Vert x^{k}-x^{k+1}\Vert\right)  \Vert
x^{k+1}-\hat{x}^{k}\Vert,
\end{align}
where $\displaystyle \underline{\tau}:=\inf_{k}\tau_{k}$, as defined in (\ref{eq:error-sum-delta-upperbound}). By
using the mean value theorem again, there is an $\check{x}^{k}$ lying in the line segment between $x^{k+1}$ and $\hat{x}^{k}$
such that%
\begin{equation}
J(x^{k+1})-J(\hat{x}^{k})=\langle\nabla J(\check{x}^{k}),x^{k+1}-\hat{x}%
^{k}\rangle. \label{equ:mean-value-theorem}%
\end{equation}
Combining (\ref{eq:descent-estimation-lowerbound-others}) and
(\ref{equ:mean-value-theorem}), yields, in view of (\ref{eq:Jinfinity}) and
(\ref{eq:JisJ}), since we are looking at $k>K_{3}>K_{2}>K_{1}$,%
\begin{align}
&\quad~  J(x^{k+1})-J^{\ast}\nonumber\\
&= J(x^{k+1})-J(\hat{x}^{k})\nonumber\\
&= \langle\nabla J(\check{x}^{k})-\nabla J(x^{k}),x^{k+1}-\hat{x}^{k}%
\rangle+\langle\nabla J(x^{k}),x^{k+1}-\hat{x}^{k}\rangle\nonumber\\
&\leq \Vert\nabla J(\check{x}^{k})-\nabla J(x^{k})\Vert\Vert x^{k+1}-\hat
{x}^{k}\Vert+\langle\nabla J(x^{k}),x^{k+1}-\hat{x}^{k}\rangle\nonumber\\
&\leq \left(  L\Vert\check{x}^{k}-x^{k}\Vert+\Vert\theta^{k}\Vert+\frac
{1}{\underline{\tau}}\Vert e^{k}\Vert+\frac{1}{\underline{\tau}}\Vert
x^{k}-x^{k+1}\Vert\right)  \Vert x^{k+1}-\hat{x}^{k}\Vert.\text{ }
\label{equ:f-and-fmin-error-upperbound-intermidiate}%
\end{align}
To finish the proof we further bound from above the right-hand side of
(\ref{equ:f-and-fmin-error-upperbound-intermidiate}). For the term
$\Vert\check{x}^{k}-x^{k}\Vert$, we note that $\check{x}^{k}$ is in the
line segment between $x^{k+1}$ and $\hat{x}^{k},$ thus,%
\begin{equation}
\Vert x^{k+1}-\check{x}^{k}\Vert+\Vert\check{x}^{k}-\hat{x}^{k}\Vert=\Vert
x^{k+1}-\hat{x}^{k}\Vert\leq\Vert x^{k}-x^{k+1}\Vert+\Vert x^{k}-\hat{x}%
^{k}\Vert,
\end{equation}
which, when combined with%
\begin{equation}
\Vert\check{x}^{k}-x^{k}\Vert\leq\Vert x^{k}-x^{k+1}\Vert+\Vert x^{k+1}
-\check{x}^{k}\Vert,
\end{equation}
and%
\begin{equation}
\Vert\check{x}^{k}-x^{k}\Vert\leq\Vert x^{k}-\hat{x}^{k}\Vert+\Vert\hat{x}%
^{k}-\check{x}^{k}\Vert,
\end{equation}
yields%
\begin{equation}\label{eq:checkxk-xk-upperbound}%
\Vert\check{x}^{k}-x^{k}\Vert\leq\Vert x^{k}-x^{k+1}\Vert+\Vert x^{k}-\hat
{x}^{k}\Vert.
\end{equation}
On the other hand, (\ref{eq:local-error-bound-application}) and
(\ref{equ:residual-upper-bound}) allows us to write%
\begin{equation}\label{equ:error-between-xn-and-solution}%
\Vert x^{k}-\hat{x}^{k}\Vert\leq\beta\eta_{2}(\Vert x^{k}-x^{k+1}\Vert+\Vert
e^{k}\Vert+\Vert\theta^{k}\Vert),\quad\text{for all }k>K_{3}.
\end{equation}
Thus, we have for the term $L\Vert\check{x}^{k}-x^{k}\Vert$, using
(\ref{eq:checkxk-xk-upperbound}) and (\ref{equ:error-between-xn-and-solution}),%
\begin{align}
L\Vert\check{x}^{k}-x^{k}\Vert &  \leq L\left(  \Vert x^{k}-x^{k+1}\Vert+\Vert
x^{k}-\hat{x}^{k}\Vert\right) \nonumber\\
&  \leq L\left(  \Vert x^{k}-x^{k+1}\Vert+\beta\eta_{2}(\Vert x^{k}%
-x^{k+1}\Vert+\Vert e^{k}\Vert+\Vert\theta^{k}\Vert)\right) \nonumber\\
&  \leq L(1+\beta\eta_{2})(\Vert x^{k}-x^{k+1}\Vert+\Vert e^{k}\Vert
+\Vert\theta^{k}\Vert).
\end{align}
For the term $\Vert x^{k+1}-\hat{x}^{k}\Vert$ in (\ref{eq:descent-estimation-lowerbound-others}), we use the triangle inequality
and (\ref{equ:error-between-xn-and-solution}) to get%
\begin{align}\label{eq:xkplus1-hatxk-upperbound}
\Vert x^{k+1}-\hat{x}^{k}\Vert &  \leq\Vert x^{k}-x^{k+1}\Vert+\Vert
x^{k}-\hat{x}^{k}\Vert\nonumber\\
&  \leq\Vert x^{k}-x^{k+1}\Vert+\beta\eta_{2}(\Vert x^{k}-x^{k+1}\Vert+\Vert
e^{k}\Vert+\Vert\theta^{k}\Vert)\nonumber\\
&  \leq(1+\beta\eta_{2})(\Vert x^{k}-x^{k+1}\Vert+\Vert e^{k}\Vert+\Vert
\theta^{k}\Vert).
\end{align}
Finally, the term $\displaystyle\Vert\theta^{k}\Vert+\displaystyle\frac
{1}{\underline{\tau}}\Vert e^{k}\Vert+\displaystyle\frac{1}{\underline{\tau}%
}\Vert x^{k}-x^{k+1}\Vert$ in the right-hand side of
(\ref{equ:f-and-fmin-error-upperbound-intermidiate}) can also be bounded above by
\begin{equation}
\displaystyle\Vert\theta^{k}\Vert+\displaystyle\frac{1}{\underline{\tau}}\Vert
e^{k}\Vert+\displaystyle\frac{1}{\underline{\tau}}\Vert x^{k}-x^{k+1}\Vert
\leq(1+\displaystyle\frac{1}{\underline{\tau}})(\Vert x^{k}-x^{k+1}\Vert+\Vert
e^{k}\Vert+\Vert\theta^{k}\Vert).
\end{equation}
Defining
\begin{equation}
\displaystyle\eta_{3}:=(L+L\beta\eta_{2}+1+\frac{1}{\underline{\tau}}%
)(1+\beta\eta_{2}),
\end{equation}
and using all the bounds from above, i.e., (\ref{equ:f-and-fmin-error-upperbound-intermidiate}), (\ref{eq:checkxk-xk-upperbound}), (\ref{equ:error-between-xn-and-solution}) and (\ref{eq:xkplus1-hatxk-upperbound}), we obtain%
\begin{equation}
J(x^{k+1})-J^{\ast}\leq\eta_{3}\left(  \Vert x^{k}-x^{k+1}\Vert+\Vert
e^{k}\Vert+\Vert\theta^{k}\Vert\right)  ^{2},\text{ for all }k>K_{3},
\end{equation}
which completes the proof.
\end{proof}

Combining Theorem \ref{theorem:f-value-converge}, Proposition \ref{proposition:sequence-error-estimation} and Proposition \ref{proposition:f-and-fmin-error-upperbound}, it can be seen that $\lim_{k \rightarrow \infty} J(x^k) = J^{\ast}$. As an immediate application of the Proposition
\ref{proposition:f-and-fmin-error-upperbound}, we get the following
intermediate proposition that leads to the final result.

\begin{proposition}
\label{proposition:convergence-of-series-omegan}Under the conditions of
Proposition \ref{proposition:descent-condition-bound}, if $\{x^{k}%
\}_{k=0}^{\infty}$ is any sequence generated by the PSG method with bounded outer
perturbations of (\ref{eq:PSG-with-summable-perturbations}), and if $\displaystyle\lambda
_{k}:=\sqrt{J(x^{k})-J^{\ast}}$ for all $k\geq0$, then $\displaystyle\sum
_{k=0}^{\infty}\lambda_{k}<+\infty$.
\end{proposition}

\begin{proof}
There exist real numbers $0<\eta_{4}<1$ and $\eta_{5}>0$ such that%
\begin{equation}\label{eq:f-and-fmin-error-root-relation}%
\sqrt{J(x^{k+1})-J^{\ast}}\leq\eta_{4}\sqrt{J(x^{k})-J^{\ast}}+\eta_{5}(\Vert
e^{k}\Vert+\Vert\theta^{k}\Vert).
\end{equation}
To prove this claim, we use $(a+b)^{2}\leq 2(a^{2}+b^{2})$ and (\ref{equ:f-and-fmin-error-upperbound})
to get%
\begin{align}
J(x^{k+1})-J^{\ast }
&\leq \eta _{3}\left( \Vert x^{k}-x^{k+1}\Vert +\Vert
e^{k}\Vert +\Vert \theta ^{k}\Vert \right) ^{2} \nonumber\\
&\leq 2\eta _{3}\Vert
x^{k}-x^{k+1}\Vert ^{2}+2\eta _{3}(\Vert e^{k}\Vert +\Vert \theta ^{k}\Vert
)^{2},  \label{eq:f-and-fmin-error-upperbound-extended}
\end{align}%
then apply (\ref{equ:sequence-error-square-upperbound}),%
with added and subtracted $J^{\ast },$ to obtain%
\begin{align}
J(x^{k+1})-J^{\ast }& \leq \frac{4\eta _{3}}{\eta _{1}}\left(
J(x^{k})-J^{\ast }\right) -\frac{4\eta _{3}}{\eta _{1}}\left(
J(x^{k+1})-J^{\ast }\right) +\frac{2\eta _{3}}{\eta _{1}^{2}}\Vert \delta
^{k}\Vert ^{2}  \notag \\
& ~~~+2\eta _{3}(\Vert e^{k}\Vert +\Vert \theta ^{k}\Vert )^{2}.
\end{align}%
Rearranging terms yields%
\begin{eqnarray}
J(x^{k+1})-J^{\ast } &\leq &\frac{4\eta _{3}}{\eta _{1}+4\eta _{3}}\left(
J(x^{k})-J^{\ast }\right) +\frac{2\eta _{3}}{\eta _{1}(\eta _{1}+4\eta _{3})}%
\Vert \delta ^{k}\Vert ^{2}  \notag \\
&&+\frac{2\eta _{1}\eta _{3}}{\eta _{1}+4\eta _{3}}(\Vert e^{k}\Vert +\Vert
\theta ^{k}\Vert )^{2}.
\end{eqnarray}%
On the other hand, (\ref{eq:error-sum-delta-upperbound}) leads to%
\begin{equation}
\Vert \delta ^{k}\Vert ^{2}\leq \frac{1}{\underline{\tau }^{2}}\Vert
e^{k}\Vert ^{2}+\Vert \theta ^{k}\Vert ^{2}\leq \frac{1}{\hat{\tau}^{2}}%
\left( \Vert e^{k}\Vert ^{2}+\Vert \theta ^{k}\Vert ^{2}\right)
\label{eq:error-sum-delta-upperbound-extended}
\end{equation}%
with $\displaystyle\underline{\tau }:=\inf_{k}\tau _{k}$
and $\displaystyle\hat{\tau}:=\min \{1,\inf_{k}\tau _{k}\}>0$ as defined
earlier. 
Therefore,%
\begin{eqnarray}
J(x^{k+1})-J^{\ast } &\leq &\frac{4\eta _{3}}{\eta _{1}+4\eta _{3}}\left(
J(x^{k})-J^{\ast }\right)  \notag \\
&&+\left( \frac{2\eta _{3}}{\eta _{1}(\eta _{1}+4\eta _{3})}\frac{1}{\hat{%
\tau}^{2}}+\frac{2\eta _{1}\eta _{3}}{\eta _{1}+4\eta _{3}}\right) (\Vert
e^{k}\Vert +\Vert \theta ^{k}\Vert )^{2}.
\end{eqnarray}%
Using $\displaystyle\sqrt{a+b}\leq \sqrt{a}+\sqrt{b}$ gives%
\begin{eqnarray}
\sqrt{J(x^{k+1})-J^{\ast }} &\leq &\sqrt{\frac{4\eta _{3}}{\eta _{1}+4\eta
_{3}}}\sqrt{J(x^{k})-J^{\ast }}  \notag \\
&&+\sqrt{\frac{2\eta _{3}}{\eta _{1}(\eta _{1}+4\eta _{3})}\frac{1}{\hat{\tau%
}^{2}}+\frac{2\eta _{1}\eta _{3}}{\eta _{1}+4\eta _{3}}}\left( \Vert
e^{k}\Vert +\Vert \theta ^{k}\Vert \right) .
\end{eqnarray}%
Denoting $\displaystyle\eta _{4}:=\sqrt{\frac{\displaystyle4\eta _{3}}{%
\displaystyle\eta _{1}+4\eta _{3}}}$ and $\displaystyle\eta _{5}:=\sqrt{%
\frac{\displaystyle2\eta _{3}}{\displaystyle\eta _{1}(\eta _{1}+4\eta _{3})}%
\frac{\displaystyle1}{\displaystyle\hat{\tau}^{2}}+\frac{\displaystyle2\eta
_{1}\eta _{3}}{\displaystyle\eta _{1}+4\eta _{3}}}$, we obtain (\ref{eq:f-and-fmin-error-root-relation}) and,
from the definition of $\eta _{4}$ and the fact that $\eta _{1}>0,\eta
_{3}>0 $,
\begin{equation}
0<\eta _{4}<1.
\end{equation}
It follows from (\ref{eq:f-and-fmin-error-root-relation}) that
\begin{equation}
\lambda_{k+1}\leq\eta_{4}\lambda_{k}+\eta_{5}(\Vert e^{k}\Vert+\Vert\theta
^{k}\Vert). \label{equ:sequence-omega-relation}%
\end{equation}
Then, for all $M>N$,
\begin{align}
\sum_{k=N+1}^{M}\lambda_{k}
&=\sum_{k=N}^{M-1}\lambda_{k+1}\nonumber\\
&\leq\eta_{4}
\sum_{k=N}^{M-1}\lambda_{k}+\eta_{5}\sum_{k=N}^{M-1}(\Vert e^{k}\Vert
+\Vert\theta^{k}\Vert)\nonumber\\
&  \leq\eta_{4}\lambda_{N}+\eta_{4}\sum_{k=N+1}^{M}\lambda_{k}+\eta_{5}%
\sum_{k=N}^{M}(\Vert e^{k}\Vert+\Vert\theta^{k}\Vert).
\end{align}
Consequently,
\begin{equation}
\sum_{k=N+1}^{M}\lambda_{k}\leq\frac{\eta_{4}}{1-\eta_{4}}\lambda_{N}+\frac
{\eta_{5}}{1-\eta_{4}}\sum_{k=N}^{M}(\Vert e^{k}\Vert+\Vert\theta^{k}\Vert).
\end{equation}
And hence,
\begin{equation}
\sum_{k=N+1}^{\infty}\lambda_{k}\leq\frac{\eta_{4}}{1-\eta_{4}}\lambda_{N}%
+\frac{\eta_{5}}{1-\eta_{4}}\sum_{k=N}^{\infty}(\Vert e^{k}\Vert+\Vert
\theta^{k}\Vert).
\end{equation}
The proof now follows by (\ref{eq:summable-perturbations}), (\ref{eq:grad-dev-summable}).
\end{proof}

Finally, we are ready to prove that sequences generated by the PSG method with
bounded outer perturbations of (\ref{eq:PSG-with-summable-perturbations})
converge to a stationary point in $S$. We do this by combining Proposition
\ref{proposition:sequence-error-estimation}, Proposition
\ref{proposition:residual-upper-bound} and Proposition
\ref{proposition:convergence-of-series-omegan}.

\begin{theorem}
\label{theorem:convergence-sequence-points}Under the conditions of Proposition
\ref{proposition:descent-condition-bound}, if $\{x^{k}\}_{k=0}^{\infty}$ is
any sequence generated by the PSG method with bounded outer perturbations of
(\ref{eq:PSG-with-summable-perturbations}), then it converges to a stationary point of the
problem (\ref{eq:convex-minimization}), i.e., to a point in $S$.
\end{theorem}

\begin{proof}
Obviously,%
\begin{align}
|J(x^{k})-J(x^{k+1})|^{1/2}  &\leq \left(  |J(x^{k})-J^{\ast}|+|J(x^{k+1}%
)-J^{\ast}|\right)  ^{1/2}\nonumber\\
&\leq \lambda_{k}+\lambda_{k+1},
\end{align}
which implies, by Proposition \ref{proposition:convergence-of-series-omegan},
that%
\begin{equation}
\sum_{k=0}^{\infty}|J(x^{k})-J(x^{k+1})|^{1/2}<+\infty.
\end{equation}
This, along with Proposition \ref{proposition:sequence-error-estimation},
guarantees that%
\begin{equation}
\sum_{k=0}^{\infty}\Vert x^{k}-x^{k+1}\Vert<+\infty,
\end{equation}
which implies that the sequence $\{x^{k}\}_{k=0}^{\infty}$ generated by (\ref{eq:PSG-with-summable-perturbations})--(\ref{eq:grad-dev-summable}) converges. Denoting $x^{\ast}:=\lim_{k\rightarrow\infty}x^{k}$ and using Proposition
\ref{proposition:residual-upper-bound} we get from
(\ref{equ:residual-upper-bound}) that $\Vert r(x^{\ast})\Vert=0$, i.e.,
$x^{\ast}\in S$, and the proof is complete.
\end{proof}

\section{Bounded Perturbation Resilience of PSG Methods\label{sect:BPR-PSG}}

In this section, we prove the bounded perturbation resilience (BPR) of PSG
methods. This property is fundamental for the application of the
superiorization methodology (SM) to them. We do this by establishing a
relationship between BPR and bounded outer perturbations given by
(\ref{eq:PSG-with-summable-perturbations-introduction}%
)--(\ref{eq:summable-perturbations-introduction}).

\subsection{Bounded Perturbation Resilience\label{subsect:BPR}}

The superiorization methodology (SM) of
\cite{Censor2010Perturbation,Censor2013Projected,Herman2012Superiorization} is
intended for nonlinear constrained minimization (CM) problems of the form:%
\begin{equation}
\label{eq:constrained-minimization}\mathrm{minimize}\left\{  \phi(x)\mid
x\in\Psi\right\}  ,
\end{equation}
where $\phi:\mathbb{R}^{n}\rightarrow\mathbb{R}$ is an objective function and
$\Psi\subseteq{\mathbb{R}^{n}}$ is the \textit{solution set} of another
problem. The set $\Psi$ could be the solution set of a \textit{convex
feasibility problem} (CFP) of the form: find a vector $x^{\ast}\in\Psi
:=\cap_{i=1}^{I}C_{i},$ where the sets $C_{i}\subseteq\mathbb{R}^{n}$ ($1 \leq
i \leq I$) are closed convex subsets of the Euclidean space $\mathbb{R}^{n}$,
see, e.g.,
\cite{Bauschke1996projection,Byrne2008Applied,Chinneck2008Feasibility} or
\cite[Chapter 5]{Censor1997Parallel} for results and references on this broad
topic. In such a case we deal in (\ref{eq:constrained-minimization}) with a
standard CM problem. Here we are interested in the case wherein $\Psi$ is the
solution set of another CM, namely the one presented at the beginning of the
paper,%
\begin{equation}
\label{eq:CM-for-solution-set}\mathrm{minimize}\left\{  J(x)\mid x\in
\Omega\right\}  ,
\end{equation}
i.e., we wish to look at,%
\begin{equation}
\label{eq:constraint-by-solution-set}\Psi:=\left\{  x^{\ast}\in\Omega\mid
J(x^{\ast})\leq J(x)\text{ for all }x\in\Omega\right\}  ,
\end{equation}
assuming that $\Psi$ is nonempty.

In either case, or any other case of the set $\Psi$, the SM strives not to
solve (\ref{eq:constrained-minimization}) but rather the task is to find a
point in $\Psi$ that is \textit{superior} (i.e., has a lower, but not
necessarily minimal, value of the $\phi$ objective function value) to one
returned by an algorithm that solves (\ref{eq:CM-for-solution-set}) alone.
This is done in the SM by first investigating the bounded perturbation
resilience of an algorithm designed to solve (\ref{eq:CM-for-solution-set})
and then proactively using such permitted perturbations in order to steer the
iterates of such an algorithm toward lower values of the $\phi$ objective
function while not loosing the overall convergence to a point in $\Psi$. See
\cite{Censor2010Perturbation,Censor2013Projected,Herman2012Superiorization}
for details of the SM. A recent review of superiorization-related previous
work appears in \cite[Section 3]{Censor2013Projected}.

In this paper we do not perform superiorization of any algorithm. Such
superiorization of the EM algorithm with total variation (TV) serving as the
$\phi$ objective function and an application of the approach to an inverse
problem of image reconstruction for bioluminescence tomography will be
presented in a sequel paper. Our aim here is to pave the way for such an
application by proving the bounded perturbation resilience that is needed in
order to do superiorization.

For technical reasons that will become clear as we proceed, we introduce an
additional set $\Theta$ such that $\Psi\subseteq\Theta\subseteq{\mathbb{R}%
^{n}}$ and assume that we have an \textit{algorithmic operator}
$\boldsymbol{A}_{\Psi}:\mathbb{R}^{n}\rightarrow\Theta$, that defines a
\textit{Basic Algorithm} as follows.

\begin{algorithm}
\label{alg:basic} \textbf{The Basic Algorithm}

\textbf{Initialization}: $x^{0}\in\Theta$ is arbitrary;

\textbf{Iterative Step}: Given the current iterate vector $x^{k}$, calculate
the next iterate $x^{k+1}$ by%
\begin{equation}
\label{eq:basic-algorithm}x^{k+1}=\boldsymbol{A}_{\Psi}\left(  x^{k}\right)  .
\end{equation}

\end{algorithm}

The bounded perturbation resilience (henceforth abbreviated by BPR) of such a
basic algorithm is defined next.

\begin{definition}
\label{definition-BPR}\textbf{\textit{Bounded Perturbation Resilience (BPR)}}
An algorithmic operator $\boldsymbol{A}_{\Psi}:\mathbb{R}^{n}\rightarrow\Theta$ is said
to be \textit{bounded perturbations resilient} if the following holds. If
Algorithm \ref{alg:basic} generates sequences $\{x^{k}\}_{k=0}^{\infty}$ with
$x^{0}\in\Theta,$ that converge to points in $\Psi,$ then any sequence
$\{y^{k}\}_{k=0}^{\infty}$, starting from any $y^0 \in \Theta$, generated by%
\begin{equation}
y^{k+1}=\boldsymbol{A}_{\Psi}\left(  y^{k}+ \beta_k v^{k}\right)  ,\text{ for all
}k\geq0,\label{eq:super-perturbed}%
\end{equation}
where ($i$) the vector sequence $\{v^{k}\}_{k=0}^{\infty}$
is bounded, and ($ii$) the scalars $\{\beta_{k}\}_{k=0}^{\infty}$ are such that
$\beta_{k}\geq0$ for all $k\geq0$, and
${\sum_{k=0}^{\infty}} \beta_{k}<\infty,$ and ($iii$) $y^{k}+ \beta_k v^{k}\in\Theta$ for all $k\geq0,$
also converges to a point in $\Psi$.
\end{definition}

Comparing this definition with \cite[Definition 1]{Censor2010Perturbation},
\cite[Subsection II.C]{Herman2012Superiorization} and \cite[Definition
4.2]{Censor2013Projected}, we observe that ($iii$) in Definition
\ref{definition-BPR} above is needed only if $\Theta\neq\mathbb{R}^{n}$. In
that case, the condition ($iii$) of Definition \ref{definition-BPR} above is
enforced in the superiorized version of the basic algorithm, see step (xiv) in
the \textquotedblleft Superiorized Version of Algorithm P\textquotedblright%
\ in \cite[p. 5537]{Herman2012Superiorization} and step (14) in
\textquotedblleft Superiorized Version of the ML-EM
Algorithm\textquotedblright\ in \cite[Subsection II.B]%
{Garduno2013Superiorization}. This will be the case in the present work.

An important special case, from which the superiorization methodology
originally grew and developed, is when $\Psi$ is the solution set of the
(linear) convex feasibility problem and $A_{\Psi}$ is a string-averaging
projection method. This was discussed and experimented with for problems of
image reconstruction from projections wherein the function $\phi$ of
(\ref{eq:constrained-minimization}) was the total variation (TV) of the image
vector $x,$ see \cite{Butnariu2007Stable,Davidi2009Perturbation}.

Note also that in later works
\cite{Censor2013Projected,Herman2012Superiorization} the notion of BPR was
replaced by that of \textit{strong perturbation resilience} which caters to
situations where $\Psi$ might be empty, however we still work here with the
above asymptotic notion of BPR and assume that $\Psi$ is nonempty. Treating
the PSG method as the Basic Algorithm $A_{\Psi}$, our strategy was to first
prove convergence of the PSG iterative algorithm with bounded outer
perturbations, i.e., convergence of
\begin{equation}
x^{k+1}=P_{\Omega}(x^{k}-\tau_{k}D(x^{k})\nabla J(x^{k})+e^{k}).
\label{eq:PSG-with-outer-perturbations}%
\end{equation}
We show next how the convergence of this yields BPR according to Definition
\ref{definition-BPR}. Such a two steps strategy was also applied in \cite[p.
541]{Butnariu2007Stable}.

A superiorized version of any Basic Algorithm\ employs the perturbed version
of the Basic Algorithm as in (\ref{eq:super-perturbed}). A
certificate to do so in the superiorization method, see \cite{censor-weak-14},
is gained by showing that the Basic Algorithm is BPR (or strongly perturbation
resilient, a notion not discussed in the present paper). Therefore, proving
the BPR of an algorithm is the first step toward superiorizing it. This is
done for the PSG method in the next subsection.

\subsection{The BPR of PSG Methods as a Consequence of Bounded Outer
Perturbation Resilience}

In this subsection, we prove the BPR of the PSG method whose iterative step is
given by (\ref{eq:PSG}). To this end we treat the right-hand side of
(\ref{eq:PSG}) as the algorithmic operator $\boldsymbol{A}_{\Psi}$ of
Definition \ref{definition-BPR}, namely, we define for all $k\geq0$,
\begin{equation}
\label{eq:algorithmic-operator-realization}\boldsymbol{A}_{\Psi}\left(
x^{k}\right)  := P_{\Omega}(x^{k} -\tau_{k} D(x^{k}) \nabla J(x^{k})),
\end{equation}
and identify the solution set $\Psi$ there with the set $S$ of
(\ref{eq:stationary-points-set}), and identify the additional set $\Theta$
there with the constraint set $\Omega$ of (\ref{eq:convex-minimization}).

According to Definition \ref{definition-BPR}, we need to show convergence of
any sequence $\{x^{k}\}_{k=0}^{\infty}$ that, starting from any $x^{0}
\in\Omega$, is generated by
\begin{align}
\label{eq:super-PSG}x^{k+1}  &  = P_{\Omega}\left(  (x^{k} + \beta_{k} v^{k})
-\tau_{k} D(x^{k} + \beta_{k} v^{k} ) \nabla J(x^{k} + \beta_{k}
v^{k})\right)  ,
\end{align}
for all $k \geq0$, to a point in $S$ of (\ref{eq:stationary-points-set}),
where $\{v^{k}\}_{k=0}^{\infty}$ and $\{\beta_{k}\}_{k=0}^{\infty}$ obey the
conditions ($i$) and ($ii$) in Definition \ref{definition-BPR}, respectively,
and also ($iii$) in Definition \ref{definition-BPR} holds.

The next theorem establishes the bounded perturbation resilience of the PSG
methods. The proof idea is to build a relationship between BPR and the
convergence of PSG methods with bounded outer perturbations of
(\ref{eq:PSG-with-summable-perturbations-introduction}%
)--(\ref{eq:summable-perturbations-introduction}).

We caution the reader that we introduce below the assumption that the set
$\Omega$ is bounded. This forces us to modify the problems
(\ref{eq:WLS-minimization}) and (\ref{eq:KL-minimization}) by replacing
$\Omega_{0}$ with some bounded subset of it in order to apply our results.
While this is admittedly a mathematically weaker result than we hoped for, we
note that this would not be a harsh limitation in practical applications
wherein such boundedness can be achieved from problem-related practical considerations.

\begin{theorem}
Given a nonempty closed convex and bounded set $\Omega \subseteq \mathbb{R}^{n}$, assume that $J \in \mathcal{S}_{\mu,L}^{1,1}(\Omega)$ (i.e., $J$ obeys (\ref{eq:Lipschitz-gradient}) and (\ref{eq:stron-conv})) and there exists at least one point $x_{\Omega} \in \Omega$ such that $\Vert \nabla J(x_{\Omega} ) \Vert<+\infty$. Let $\{\tau_{k}\}_{k=0}^{\infty}$ be a sequence of positive scalars that fulfills (\ref{eq:stepsize-bound}), $\{D(x)\}_{k=0}^\infty$ be a sequence of diagonal scaling matrices that is either of form (\ref{eq:LS-scaling-matrix}) or (\ref{eq:KL-scaling-matrix-general}), and let $\{\theta^{k}\}_{k=0}^{\infty}$ be as in (\ref{eq:theta-k}) and for which (\ref{eq:grad-dev-summable}) holds. Under these assumptions, if the vector sequence $\{v^{k}\}_{k=0}^{\infty}$ is bounded and the scalars $\{\beta_{k}\}_{k=0}^{\infty}$ are such that $\beta_{k}\geq0$ for all $k\geq0$, and $\sum_{k=0}^{\infty}\beta_{k}<\infty$, then, for any $x^0 \in \Omega$, any sequence $\{x^k\}_{k=0}^\infty$, generated by (\ref{eq:super-PSG}) such that $x^k + \beta_k v^k \in \Omega$ for all $k \geq 0$, converges to a point in $S$ of (\ref{eq:stationary-points-set}).
\end{theorem}

\begin{proof}
The proof is in two steps. For the first step, we build a relationship between (\ref{eq:super-PSG}) and bounded outer perturbations of (\ref{eq:PSG-with-summable-perturbations-introduction})--(\ref{eq:summable-perturbations-introduction}). For the second step, we invoke Theorem \ref{theorem:convergence-sequence-points} and establish the convergence result.
\newline \textbf{Step 1.}  We show that any sequence generated by (\ref{eq:super-PSG}) satisfies
\begin{equation}\label{eq:super-PSG-with-outer-perturbations}
x^{k+1} = P_{\Omega} \left(x^{k} - \tau_k D(x^{k}) \nabla J(x^{k})+ e^k \right),
\end{equation}
with $\sum_{k=0}^\infty \Vert e^k \Vert < +\infty$. Since $\Omega$ is a bounded subset of $\mathbb{R}^n$, there exists a $r_{\Omega} > 0$ such that $\Omega \subseteq B( x_{\Omega}, r_{\Omega})$, where $B( x_{\Omega}, r_{\Omega}) \subseteq \mathbb{R}^n$ is a ball centered at $x_{\Omega}$ with radius  $r_{\Omega}$. Then, for any $x \in \Omega$,
\begin{align}\label{eq:x-boundness}
\Vert x - x_{\Omega} \Vert \leq r_{\Omega} ~\Rightarrow~ \Vert x \Vert \leq \Vert x_{\Omega} \Vert + r_{\Omega}.
\end{align}
The Lipschitzness of $\nabla J(x)$ on $\Omega$ and (\ref{eq:x-boundness}) imply that, for any $x \in \Omega$,
\begin{align}\label{eq:gradient-jx-boundness}
\Vert\nabla J(x) - \nabla J(x_{\Omega}) \Vert \leq L \|x - x_{\Omega} \Vert \Rightarrow  \Vert\nabla J(x)\Vert \leq \Vert\nabla J(x_{\Omega})\Vert + L r_{\Omega}.
\end{align}
Since the sequence $\{x^k\}_{k=0}^\infty$ generated by (\ref{eq:super-PSG}) is contained in $\Omega$, due to the projection operation $P_{\Omega}$, and $x^k + \beta_k v^k$ is also in $\Omega$, it holds that, for all $k\geq 0$, $x^k$ and $x^k + \beta_k v^k$ satisfy (\ref{eq:x-boundness}), and that $\nabla J(x^k)$ and $\nabla J(x^k + \beta_k v^k)$ satisfy (\ref{eq:gradient-jx-boundness}). Besides, the boundness of ${\{v^k\}}_{k=0}^\infty$ implies that there exist a $\overline{v}>0$ such that $\|v^k\| \leq \overline{v}$ for all $k\geq 0$. Therefore, we have
\begin{align}\label{eq:hatxk-xk-boundness}
\Vert \beta_k v^k\Vert \leq \overline{v}\beta_k.
\end{align}
From (\ref{eq:super-PSG}), the outer perturbation term $e^k$ of (\ref{eq:super-PSG-with-outer-perturbations}) is given by
\begin{align}\label{eq:difference-term}
e^k &= \left(x^k + \beta_k v^k  -\tau_k D(x^k + \beta_k v^k ) \nabla J( x^k + \beta_k v^k ) \right) - \left(x^k  -\tau_k D(x^k ) \nabla J(x^k ) \right) \nonumber \\
&= \beta_k v^k + \tau_k \left( D(x^k ) \nabla J(x^k ) -  D(x^k + \beta_k v^k ) \nabla J( x^k + \beta_k v^k )\right).
\end{align}
Given that $D(x)$ is either of form (\ref{eq:LS-scaling-matrix}) or (\ref{eq:KL-scaling-matrix-general}), we consider them separately. In what follows, we repeatedly use the fact that $\Vert A B x\Vert \leq \Vert A B\Vert_F \Vert x\Vert \leq \Vert A \Vert_F \Vert B\Vert_F \Vert x\Vert$ for any $A,B \in \mathbb{R}^{n \times n}$ and $x \in \mathbb{R}^n$, with $\Vert \cdot \Vert_F$ the Frobenius norm of matrix, see, e.g., \cite[Section 2.3]{Golub1996Matrix}.
\begin{description}
\item[($i$)] Assume that $D(x)$ is of form (\ref{eq:LS-scaling-matrix}), namely that $D(x) \equiv D_{\mathrm{LS}}$ for any $x$. For this case, combining (\ref{eq:difference-term}) with (\ref{eq:Lipschitz-gradient}), (\ref{eq:stepsize-bound}) and (\ref{eq:hatxk-xk-boundness}), and by the Minkowski inequality, we get
\begin{align}\label{eq:difference-term-bound-case1}
\Vert e^k \Vert &= \Vert \beta_k v^k + \tau_k D_{\mathrm{LS}}\left(\nabla J(x^k ) - \nabla J( x^k + \beta_k v^k ) \right) \Vert\nonumber \\
&\leq \Vert\beta_k v^k \Vert + \tau_k \Vert D_{\mathrm{LS}} \Vert_F \Vert \nabla J(x^k ) - \nabla J( x^k + \beta_k v^k ) \Vert \nonumber \\
&\leq \Vert\beta_k v^k \Vert + \tau_k \Vert D_{\mathrm{LS}} \Vert_F L \Vert \beta_k v^k \Vert \nonumber \\
&\leq (1 + 2 \Vert D_{\mathrm{LS}} \Vert_F) \overline{v}\beta_k.
\end{align}
\item[($ii$)]Assume that $D(x)$ is of form (\ref{eq:KL-scaling-matrix-general}), namely that $D(x) := \hat{D}X$ with $\hat{D}= \operatorname*{diag}\{1/\hat{s}_j\}$ and $X=\operatorname*{diag}\{x_j\}$ diagonal matrices. In this case, combining (\ref{eq:difference-term}) with (\ref{eq:Lipschitz-gradient}), (\ref{eq:stepsize-bound}), (\ref{eq:x-boundness}), (\ref{eq:gradient-jx-boundness}), (\ref{eq:hatxk-xk-boundness}), and by the Minkowski inequality, we get
\begin{align}
\Vert e^k \Vert
&= \Vert \beta_k v^k  + \tau_k \left(D(x^k ) \nabla J(x^k ) - D( x^k+\beta_k v^k ) \nabla J( x^k+\beta_k v^k ) \right) \Vert \nonumber\\
&=\Vert \beta_k v^k + \tau_k \left(D(x^k ) \nabla J(x^k ) - D( x^k+\beta_k v^k ) \nabla J(x^k )\right)  \nonumber\\
&\quad + \tau_k \left( D( x^k+\beta_k v^k ) \nabla J(x^k ) - D( x^k+\beta_k v^k ) \nabla J( x^k+\beta_k v^k ) \right) \Vert \nonumber\\
&\leq \Vert \beta_k v^k \Vert + \tau_k \Vert \hat{D}(X^k - \hat{X}^k) \nabla J(x^k ) \Vert \nonumber \\
&\quad + \tau_k \Vert \hat{D}\hat{X}^k ( \nabla J(x^k ) - \nabla J( x^k+\beta_k v^k ) ) \Vert \nonumber\\
&\leq  \Vert\beta_k v^k \Vert + \tau_k \Vert \hat{D} \Vert_F \Vert X^k - \hat{X}^k \Vert_F \Vert \nabla J(x^k ) \Vert + \tau_k \Vert \hat{D} \hat{X}^k \Vert_F L \Vert \beta_k v^k \Vert\nonumber \\
&\leq  (1 + \tau_k \Vert \hat{D} \Vert_F \Vert \nabla J(x^k ) \Vert + \tau_k L \Vert \hat{D} \Vert_F \Vert \hat{X}^k \Vert_F  ) \Vert \beta_k v^k \Vert \label{eq:temp-1}\\
&\leq (1 + 2 \Vert \hat{D} \Vert_F \Vert \nabla J(x^k ) \Vert / L + 2\Vert \hat{D} \Vert_F \Vert x^k + \beta_k v^k \Vert ) \overline{v}\beta_k \label{eq:temp-2}\\
&\leq \left(1 + 2 \Vert \hat{D} \Vert_F (\Vert\nabla J(x_{\Omega})\Vert /L + \Vert x_{\Omega} \Vert + 2r_{\Omega}) \right) \overline{v}\beta_k \label{eq:difference-term-bound-case2},
\end{align}
where $X^k:=\operatorname*{diag}\{x^k_j\}$, $\hat{X}^{k}:=\operatorname*{diag}\{(x^k + \beta_k v^k)_j\}$, and (\ref{eq:temp-1}) holds by the fact that $\Vert X^k - \hat{X}^k \Vert_F = \Vert x^k - (x^k + \beta_k v^k) \Vert = \Vert \beta_k v^k \Vert$, and (\ref{eq:temp-2}) holds since $\Vert \hat{X}^k \Vert_F = \Vert x^k + \beta_k v^k \Vert$, and (\ref{eq:difference-term-bound-case2}) holds by (\ref{eq:x-boundness}) and (\ref{eq:gradient-jx-boundness}).
\end{description}
Defining a constant
\begin{align}
C_{\Omega} :=
\overline{v} + 2 \overline{v} \cdot \max \left \{ \Vert D_{\mathrm{LS}} \Vert_F, ~\Vert \hat{D} \Vert_F (\Vert\nabla J(x_{\Omega})\Vert /L + \Vert x_{\Omega} \Vert + 2r_{\Omega}) \right \},
\end{align}
and considering (\ref{eq:difference-term-bound-case1}) or (\ref{eq:difference-term-bound-case2}), yields that in either case ($i$) or case ($ii$),
\begin{align}\label{eq:perturbation-upper-bound}
\Vert e^k \Vert \leq C_{\Omega} \beta_k .
\end{align}
Then, $\sum_{k=0}^\infty \beta_k < + \infty$ implies that $\sum_{k=0}^\infty \Vert e^k \Vert < +\infty$.
\newline \textbf{Step 2.} Under the given conditions, by invoking Theorem \ref{theorem:convergence-sequence-points}, we know that, for any $x^0 \in \Omega$, any sequence $\{x^k\}_{k=0}^\infty$, generated by (\ref{eq:super-PSG-with-outer-perturbations}) in which $\sum_{k=0}^\infty \Vert e^k \Vert < + \infty$, converges to a point in $S$ of (\ref{eq:stationary-points-set}). Hence, the sequence generated by (\ref{eq:super-PSG}) also converges to the same point of $S$.
\end{proof}

\section*{Acknowledgments}

We greatly appreciate the constructive comments of two anonymous reviewers and
the Coordinating Editor which helped us improve the paper. This work was
supported in part by the National Basic Research Program of China (973
Program) (2011CB809105), the National Science Foundation of China (61421062)
and the United States-Israel Binational Science Foundation (BSF) grant number 2013003.


\end{document}